\documentclass{cpaaEA} % Use the aims.cls file to compile your paper
\usepackage{amsmath,amssymb}
\usepackage{paralist}
\usepackage[misc]{ifsym}
\usepackage{epsfig} %For pictures: screened artwork should be set up with an 85 or 100 line screen
\usepackage{epstopdf} %This is to transfer .eps figure to .pdf figure; please compile your article using PDFLaTex or PDFTeXify.
\usepackage[colorlinks=true]{hyperref}
\hypersetup{urlcolor=blue, citecolor=red}
\allowdisplaybreaks

\usepackage{esint}
\usepackage{bm}
\usepackage{xfrac}
\usepackage{comment}

\def\currentvolume{X}
 \def\currentissue{X}
  \def\currentyear{200X}
   \def\currentmonth{XX}
    \def\ppages{X-XX}
     \def\DOI{10.3934/cpaa.2023110}

\numberwithin{equation}{section}
\allowdisplaybreaks[4]
\newcommand{\xqedhere}[2]{%
	\rlap{\hbox to#1{\hfil\llap{\ensuremath{#2}}}}}

\newtheorem{theorem}{Theorem}[section]
\newtheorem{corollary}[theorem]{Corollary}

\newtheorem{lemma}[theorem]{Lemma}
\newtheorem{proposition}[theorem]{Proposition}

\theoremstyle{definition}

\newtheorem{remark}[theorem]{Remark}

\newcommand{\R}{\mathbb{R}}
\newcommand{\C}{\mathbb{C}}

\newcommand{\Z}{\mathbb{Z}}
\newcommand{\T}{\mathbb{T}}

\newcommand{\p}{\partial}

\newcommand{\les}{\lesssim}

\let\div\relax
\DeclareMathOperator{\div}{div}

\DeclareMathOperator{\curl}{curl}

\let\Re\relax
\DeclareMathOperator{\Re}{Re}
\let\Im\relax
\DeclareMathOperator{\Im}{Im}

\let\tilde\relas
\newcommand{\tilde}[1]{\widetilde{#1}}
\DeclareMathOperator*{\esssup}{ess\,sup}

\title[Remarks on the complex Euler equations]
{Remarks on the complex Euler equations} % Only the first word and proper nouns should be capitalized.

\author[Dallas Albritton and W. Jacob Ogden]{}

\subjclass{Primary: 35Q31, 35Q35; Secondary: 37L40, 37D10.}
\keywords{Complex Euler equations, Euler equations, Euler-Arnold equations, ill-posedness, unstable manifold}

\thanks{DA was supported by National Science Foundation Postdoctoral Fellowship Grant No.\ 2002023. WJO was supported by National Science Foundation Graduate Research Fellowship Grant No.\ DGE-2140004.}

\thanks{$^*$Corresponding author: Dallas Albritton}

\begin{document}
\maketitle

% Enter the first author's name and email address; email addresses are required for each author.
% Use footnote notations to indicate address and affiliations with commas between numbers if more than one address applies; see below for examples.
\centerline{\scshape
Dallas Albritton$^{{\href{mailto:dalbritton@wisc.edu}{\textrm{\Letter}}}*1}$
and W. Jacob Ogden$^{{\href{mailto:wjogden@uw.edu}{\textrm{\Letter}}}1,2}$}

\medskip

{\footnotesize
% Enter the full affiliation and country name:
% Do not insert commas or periods at the end of lines.
 \centerline{$^1$Department of Mathematics, University of Wisconsin--Madison, USA}
} % Do not forget to end {\footnotesize with the sign }

\medskip

{\footnotesize
 % Enter the full affiliation and country name:
 \centerline{$^2$Department of Mathematics, University of Washington, USA}
}

\bigskip

\dedicatory{\small{To Vladim{\'i}r \v{S}ver{\'a}k, on the occasion of his 65th birthday, with gratitude and admiration}}

% The name of the handling editor will be entered by AIMS production staff.
% "Communicated by Handling Editor" is not needed for special issue.
% \centerline{(Communicated by Handling Editor)}

%%%%%%%%%%%%%%%%%%%%%%%%%%%%%%%%%%%%%%%%%%%%%%%%%%%%%%%
%             5. ABSTRACT
%%%%%%%%%%%%%%%%%%%%%%%%%%%%%%%%%%%%%%%%%%%%%%%%%%%%%%%

\begin{abstract}
We consider a complexification of the Euler equations introduced by \v{S}ver{\'a}k in~\cite{VS} which conserves energy. We prove that these complex Euler equations are nonlinearly ill-posed below analytic regularity and, moreover, we exhibit solutions which lose analyticity in finite time. Our examples are complex shear flows and, hence, one-dimensional. This motivates us to consider fully nonlinear systems in one spatial dimension which are non-hyperbolic near a constant equilibrium. We prove nonlinear ill-posedness and finite-time singularity for these models. Our approach is to construct an infinite-dimensional unstable manifold to capture the high frequency instability at the nonlinear level.
\end{abstract}

%\tableofcontents

\section{Introduction}

We consider the Cauchy problem for the \emph{complex Euler equations}
\begin{equation} \label{eq:ComplexEuler}
\tag{CE}
\left\lbrace
\begin{aligned}
    \partial_t u + \bar{u} \cdot \nabla u + (\nabla \bar{u})^T u + \nabla \pi &= 0 \\
    \div u &= 0
    \end{aligned} \right.
\end{equation}
on the $d$-dimensional torus $\T^d = (\R/2\pi \Z)^d$, $d \geq 2$. The velocity field $u : \T^d \times I \to \C^d$ is complex-valued. When $u = \bar{u}$, we have $(\nabla \bar{u})^T u = \sfrac  12 \nabla |u|^2$. Since this term may be absorbed into the pressure gradient, the system \eqref{eq:ComplexEuler} is a direct generalization of the Euler equations for ideal incompressible fluids. It was introduced by \v{S}ver{\'a}k in~\cite{VS} as a natural complexification preserving the geometric structure. % in the following way.

It is well known that the Euler equations can be viewed as geodesic equations on the group ${\rm SDiff}(M)$ of volume-preserving diffeomorphisms of a compact manifold $M$~\cite{Arnold,EbinMarsden}. The Lie algebra (tangent space at the identity) of ${\rm SDiff}(M)$ is the space of divergence-free vector fields. The group is further endowed with a right-invariant metric which, when restricted to the Lie algebra, is simply $\int u \cdot v \, dx$.

Beginning from this geometric point of view, the Euler equations can be derived from the structure of the Lie algebra of ${\rm SDiff}(M)$. The procedure is as follows:

Let $\mathsf{L}$ be a real Lie algebra with Lie bracket $[\cdot , \cdot ]$ and an inner product $(\cdot, \cdot  )$. Define a bilinear form $B: \mathsf{L} \times \mathsf{L} \to \mathsf{L}$, called the \emph{Arnold form}, by duality:
\begin{equation} 
( B(u,v ), w ) = (u, [v,w]) \, .
\end{equation} 
The \emph{Euler-Arnold equation} is 
\begin{equation} \label{eq:EulerArnold}
\p_t u = B ( u, u ) \, .
\end{equation}
It is the equation, written in the tangent space at the identity, for \emph{geodesics on a Lie group with right-invariant metric}, as described, for example, in~\cite{TaoEulerArnold}.
It is easy to verify that the energy $(u,u)$ is preserved by the evolution of \eqref{eq:EulerArnold}.

To recover the Euler equations, we take $\mathsf{L}$ to be the Lie algebra of divergence-free vector fields on $\mathbb{T}^d$ with Lie bracket %. Recall that the Lie bracket is given by 
\begin{equation} 
[u,v] = u \cdot \nabla v - v \cdot \nabla  u %\, ,
\end{equation}
%and we equip $\mathsf{L}$ with the
and equipped with the $L^2$ inner product. %usual
Then, for any $v \in \mathsf{L},$ \eqref{eq:EulerArnold} says
\begin{equation}\begin{aligned} 
\int_{\mathbb{T}^d} \p_t u \cdot v \, dx & = \int_{\mathbb{T}^d} u \cdot (u \cdot \nabla v - v \cdot\nabla  u ) \, dx \, .
\end{aligned} \end{equation} 
Integrating by parts and using the fact that $u$ and $v$ are assumed to be divergence-free yields that $u$ satisfies the Euler equations. 

The complex Euler equations~\eqref{eq:ComplexEuler} introduced by \v{S}ver{\'a}k \cite{VS} are the Euler-Arnold equation in the Lie algebra $\mathsf{L} \otimes \C$ of complex-valued divergence-free vector fields equipped with the inner product 
\begin{equation}
\langle u, v \rangle = \Re \int u \cdot \bar v \, dx \, .
\end{equation}
For completeness, we include the details of the derivation of \eqref{eq:ComplexEuler}. Write $u= u^r + i u^i$, $v= v^r + i v^i$. Then the Euler-Arnold equation~\eqref{eq:EulerArnold} says 
\begin{equation} \begin{aligned}
\int \p_t  u^r_k v^r_k + \p_tu^i_k v^i _k & = \int u^r _k( u^r_l v^r_{k,l}     -  v^r_l u^r_{k,l} -      u^i_l v^i_{k,l}     +  v^i_l u^i _{k,l}) 
\\ & \quad \quad  + u^i_k
( u^i_l v^r_{k,l} - v^r_l u^i_{k,l}     +  u^r_l v^i_{k,l}     -  v^i_l u^r_{k,l})
\\ & = \int - ( u^r_k u^r_l ) _{,l} v^r_k - u_k^r u_{k,l}^r v^r_l + ( u_k^r u^i_l)_{,l} v_k^i  + u^r_k  u^i _{k,l} v^i_l 
\\& \quad \quad -( u_k^i u_l^i ) _{,l} v_k^r - u^i_k u^i_{k,l} v^r_l - ( u^i_k u^r_l)_{,l} v^i_k - u^i_k u^r_{k,l} v^i_l \\  
& = \int - u^r_{k,l} u^r_l  v^r_k -  u_k^r u_{k,l}^r v^r_l + u_{k,l} ^r u^i_l v_k^i + u^r_k  u^i _{k,l} v^i_l 
\\
& \quad \quad  - u_{k,l}^i u_l^i v_k^r - u^i_k u^i_{k,l} v^r_l - u^i_{k,l} u^r_l v^i_k - u^i_k u^r_{k,l} v^i_l  
\\
& = \int - u^r_{k,l} u^r_l  v^r_k -  u_l^r u_{l,k}^r v^r_k + u_{k,l} ^r u^i_l v_k^i + u^r_l  u^i _{l,k} v^i_k 
\\
& \quad \quad   - u_{k,l}^i u_l^i v_k^r - u^i_l u^i_{l,k} v^r_k - u^i_{k,l} u^r_l v^i_k - u^i_l u^r_{l,k} v^i_k .
\end{aligned} \end{equation}
Grouping the $v^r$ and $v^i$ terms, we have 
\begin{equation} \begin{aligned}
\int  ( \partial _t u^r_k + u^r_l u^r_{k,l} + u^i_l u^i_{k,l} + u^r_l u^r_{l,k} + u^i_l u^i_{l,k}) v_k^r  &=0,\\
\int  ( \partial _t u^i_k+ u_l^r u^i_{k,l} - u^i _l u^r_{k,l} + u^i_l u^r_{ l,k} - u^r_l u^i_{l,k} ) v^i_k  & =0. 
\end{aligned} \end{equation}
Since $v$ is divergence-free, the Helmholtz-Hodge decomposition implies that $\p_t u + \bar u \cdot \nabla u + ( \nabla \bar u ) ^ T u  $ is a gradient. Hence, we obtain \eqref{eq:ComplexEuler}.

The complex Euler equations conserve the energy $\sfrac{1}{2} \int  |u|^2 \, dx$, whereas energy conservation fails when the `real' Euler equations are considered for complex-valued vector fields, indicating that \eqref{eq:ComplexEuler} is a more natural complexification of the Euler equations. {\v S}ver\'ak proposed the geometric complexifications of the Euler and Navier-Stokes equations as toy models in the regularity theory for fluid PDEs. Model equations, notably, Tao's Eulerian~\cite{TaoAveragedNSE} and Lagrangian~\cite{TaoLagrangianModifications} modifications, the generalized SQG~\cite{constantin1994singular} and Constantin-Lax-Majda/De Gregorio~\cite{constantin1985simple,de1990one} equations, and more~\cite{VS}, have led to many interesting insights. While complex-valued solutions of the Navier-Stokes~\cite{LiSinai} and viscous Burgers equations~\cite{polavcik2008zeros,li2008complex,sverak2022singularities} are known to exhibit finite-time blow-up, the analogous problem for the complex Navier-Stokes equations
\begin{equation}
	\label{eq:CNS}
\tag{CNS}
\left\lbrace
\begin{aligned}
    \partial_t u + \bar{u} \cdot \nabla u + (\nabla \bar{u})^T u + \nabla \pi &= \Delta u \\
    \div u &= 0
    \end{aligned} \right.
\end{equation}
remains open. \eqref{eq:CNS} is globally well-posed in $d=2$, and it is expected that partial regularity holds in $d=3$~\cite{ckn} with minor technical adjustments because the nonlinearity is not in divergence form.\footnote{In particular, we suspect that the partial regularity theory for~\eqref{eq:CNS} could borrow techniques from the boundary partial regularity theory~\cite{seregin2002local}. The starting point is the local energy equality, see~\eqref{eq:localenergyequality} below.}
Furthermore, the vorticity equation~\eqref{eq:2dvorticity} in two dimensions does not have a maximum principle, and the global well-posedness of~\eqref{eq:CNS} with hypodissipation $-(-\Delta)^s$, $s \in (1,2)$, is open in $d=2$. In this setting, the second author established local-in-time well-posedness in subcritical Sobolev spaces and global-in-time existence of weak solutions~\cite{Ogden}. Below we focus only on the inviscid model. 

%Just as with the real Euler equation, the energy $\sfrac{1}{2} \int  |u|^2 \, dx$ is conserved. Energy conservation fails if the real Euler equations are considered for complex-valued vector fields, indicating that \eqref{eq:ComplexEuler} is a more natural complexification of the Euler equations. {\v S}ver\'ak proposed this geometric complexification of the Euler and Navier-Stokes equations in hopes of gaining insight into the regularity problems for the real Euler and Navier-Stokes equations. Notably, Li and Sinai \cite{LiSinai} showed that complex-valued solutions of the standard Navier-Stokes equations exhibit finite-time blowup. In \cite{Ogden}, the second author established local-in-time well-posedness for the geometrically complexified Navier-Stokes equations in subcritical Sobolev spaces. The question of singularity formation from smooth initial data for the complexified Navier-Stokes equations remains open. Below we focus only on the inviscid model. Some notable examples include Tao's Eulerian~\cite{TaoAveragedNSE} and Lagrangian~\cite{TaoLagrangianModifications} modifications, Hou-Luo scenario, the generalized Constantin-Lax-Majda, etc., many of which are reviewed in~\cite{VS}. \dacomment{still needs work, DA will take a stab}

In two dimensions, the vorticity $\omega = \curl u = \nabla^\perp \cdot u$ satisfies the equation
\begin{equation}
	\label{eq:2dvorticity}
    \p_t \omega + \bar{u} \cdot \nabla \omega = 0 \, ,
\end{equation}
and $u$ can be recovered, modulo the zeroth Fourier mode, which evolves via~\eqref{eq:zerothmode}, from $\omega$ by the Biot-Savart law
\begin{equation} 
\Delta \psi = \omega \, , \quad u = \nabla^\perp \psi \, .
\end{equation} 
We consider the linearized vorticity equation near a steady solution $u \equiv a\in \C^2$,
\begin{equation}
    \p_t \omega + \bar{a} \cdot \nabla \omega = 0 \, .
\end{equation}
The solution for the $k^{\rm th}$ Fourier mode $\hat \omega_k$ is
\begin{equation}
	\label{eq:instabilityinvorticityequation}
   \hat  \omega_k(t) = e^{-i \bar{a} \cdot k t} \hat \omega_k(0) \, .
\end{equation}
When $\Re(- i\bar{a} \cdot k )> 0$, the solution grows exponentially. This simple computation, proposed to us by V.~\v{S}ver{\'a}k and already observed in \cite{Ogden}, suggests that the complex Euler equations should be \emph{ill-posed below analytic regularity}. %implies that the linearized system is \emph{ill-posed below analytic regularity}.
%The instability of the linearized equation was ; we learned it directly from 

In this paper, we prove that the complex Euler equations~\eqref{eq:ComplexEuler} are nonlinearly ill-posed in every Sobolev and Gevrey space below analytic regularity. Not only do we demonstrate norm inflation in arbitrarily short time; we prove finite-time loss of analyticity in arbitrarily short time from arbitrarily small initial data. %As we explain,
As we explain, the ill-posedness is already visible at the level of shear flows. This partially explains instabilities encountered by the second author in %attempts to conduct
numerical simulations of the complex Euler equations during the preparation of~\cite{Ogden}. %We also construct solutions of $\eqref{eq:ComplexEuler}$ which exhibit finite-time blowup.

From a certain perspective, ill-posedness is somewhat surprising; since energy is conserved, one might expect to commute derivatives $\p^\alpha$ through the equation, obtain energy estimates for $\p^\alpha u$, and close \emph{a priori} estimates at finite regularity. This argument, which is valid for hyperbolic equations, does not work for~\eqref{eq:ComplexEuler}.

Burgers equation can be considered as the geodesic equation on the group of diffeomorphisms ${\rm Diff}(\T)$. The above complexification procedure produces
\begin{equation}
	\label{eq:complexBurgers}
\p_t u + \bar{u} \p_x u + 2 u \p_x \bar{u} = 0 \, .
\end{equation}
When $u = \bar{u}$, the nonlinearity evidently becomes $3 u \p_x u$, and~\eqref{eq:complexBurgers} is therefore a direct generalization of Burgers equation for which smooth solutions also conserve energy. It was introduced by {\v S}ver{\'a}k in~\cite{VS}. Interestingly,~\eqref{eq:complexBurgers} is hyperbolic; it can be written as a first-order $2 \times 2$ quasilinear system for $(a,b) = (\Re u, \Im u)$:
\begin{equation}
\begin{aligned}
    \p_t a + 3a\p_x a + 3b\p_x b &= 0 \\
    \p_t b + b\p_x a - a\p_x b &= 0 \, .
    \end{aligned}
\end{equation}
The matrix $A = \begin{bmatrix} 3a & 3b \\ b & -a \end{bmatrix}$ has characteristic equation
\begin{equation}
    \lambda^2 - 2a \lambda - 3(a^2 + b^2) = 0
\end{equation}
and real eigenvalues
\begin{equation}
    \lambda = {a} \pm \sqrt{ 4 a^2 + 3 b^2}
\end{equation}
which are distinct except at the `umbilical point' $(a,b) = 0$, in which case the matrix $A = 0$. The system is uniquely solvable in the $C^1$ class~\cite{friedrichs1948nonlinear,Douglis1952}.\footnote{Furthermore, due to the energy conservation,~\eqref{eq:ComplexEuler} and~\eqref{eq:complexBurgers} can be shown to satisfy a weak-strong uniqueness principle.} %will exhibit weak-strong uniqueness when the `strong' solution belongs to $L^\infty_t \rm{Lip}_x$ and the `weak solution' belongs to $L^\infty_t L^2_x \cap ?$ is assumed to satisfy energy inequality.}

On the other hand, the na{\"i}vely complexified Burgers equation
\begin{equation}
	\label{eq:naiveBurgers}
    \p_t u + 3 u \p_x u = 0 \, ,
\end{equation}
where $u : \T \times I \to \C$ is complex-valued, is ill-posed below analytic regularity. This is well known in the community around non-hyperbolicity. That is, the geometric complexification~\eqref{eq:complexBurgers} of Burgers is better behaved than the na{\"i}ve complexification, in contrast to the situation for the complex Euler equations.

In Section \ref{sec:nonhyperbolicsystems}, we prove ill-posedness and finite-time loss of analyticity in a general class of fully non-linear systems
\begin{equation}
	\label{eq:systemsforillposedness}
\p_t u + F(u,\p_x u) = 0 \, , \quad u : \T \times I \to \R^m \, ,
\end{equation}
assumed to be non-hyperbolic at a single constant equilibrium, see \textbf{(A1)}-\textbf{(A2)} in Section~\ref{sec:nonhyperbolicsystems}. Our perspective is to construct the \emph{infinite-dimensional unstable manifold} associated to the linear instability. We are partially inspired by an analogous construction of stable manifolds for vortex sheets due to Duchon and Robert~\cite{DuchonRobertVortexSheet}, see also~\cite{caflisch1986long,caflisch1989singular}. %, for example, in our decision to work in the Wiener algebra instead of $L^2$-based analytic spaces.
For more about ill-posedness in fluid PDEs, see~\cite{desjardins2006nonlinear,Wu2006,han2016ill,ambrose2019global} and the invariant manifold constructions in~\cite{de2009smooth,cheng2020stable}. %~\cite{Ambrosesurvey}.

Our original intention was to construct the full unstable manifold associated to the instability~\eqref{eq:instabilityinvorticityequation} in the complex Euler equations~\eqref{eq:ComplexEuler}. We encountered the problem that the associated semigroups do not smooth in the $a^\perp$ direction. It would be interesting to know whether the construction is possible. 

Ill-posedness for non-hyperbolic systems $\p_t u + F(t,x,u,\p_x u) = 0$, especially in the quasilinear case $\p_t u + A(t,x,u) \p_x u = F(t,x,u)$, is old and well studied. A seminal work in this direction is M{\'e}tivier's~\cite{metivier2005remarks}; see also recent works~\cite{lerner2010instability,morisse2020hyperbolicity} and references therein. The onset of instability, where solutions evolve from hyperbolic to elliptic regions, is also understood~\cite{lerner2018onset,morisse2018hyperbolicity,ndoumajoud2020m,ndoumajoud2021m}. Because the constructions are extremely general, they are generally short-time and microlocal. % in nature.
While we do not claim great novelty, our construction, for the restricted systems~\eqref{eq:systemsforillposedness} with a constant equilibrium, produces not only instability but also solutions on $\T \times (-\infty,0)$ which lose analyticity at time zero. Our construction is local in frequency only and, we hope, comparatively simple and transparent.

%Metivier, Desjardin-Grenier (this uses Grenier's technique of integrating forward), Han-Kwan--Nguyen, ... most common example is complex Burgers or complex Burgers starting from real but with i forcing; notably, these do not use the forward-backward integration; for completeness, give our own proof which emphasises the dynamical point of view; goes back to Duchon and Robert, and related to Euler (vortex sheets); some results on non-existence of sols known)

\section{Properties of the complex Euler equations}
\label{sec:simpleproperties}

We now discuss some simple properties of \eqref{eq:ComplexEuler}.

%As mentioned above,

Conservation of energy is an immediate consequence of the geometric structure of the Euler-Arnold equation~\eqref{eq:EulerArnold}, and it is also straightforward to derive from the PDE. If we recognize $(\nabla \bar{u})^T u = \sfrac{1}{2} \nabla |u|^2 + \Im \, (\nabla \bar{u})^T u$ and write $\pi = p + \sfrac{1}{2} \nabla |u|^2$, then smooth solutions of~\eqref{eq:ComplexEuler} satisfy the \emph{local energy equality}
\begin{equation}
	\label{eq:localenergyequality}
\p_t \frac{1}{2}  |u|^2 + \div \Re \bigg[ u \big( \frac{1}{2}|u|^2 + \bar{p} \big) \bigg]  = 0 \, .
\end{equation}
The pressure $p$ satisfies
\begin{equation}
-\Delta p = \div \div u \otimes \bar{u} + \div \Im [ (\nabla \bar{u})^T u ] \, .
\end{equation}

\begin{comment}
While conservation of energy is an immediate consequence of the geometric structure of the Euler-Arnold equation \eqref{eq:EulerArnold}, it is also straightforward to derive %energy conservation
directly from the PDE. Recall that 
%We have 
\begin{equation} \begin{aligned}
\frac{d}{dt} \frac 12 \int \bar u \cdot u = \frac 12 \int \bar u \cdot \p_t u + u \cdot \p_t \bar u  = \Re \int \bar u \cdot \p_t u \, .
\end{aligned} \end{equation} 
Using the equation~\eqref{eq:ComplexEuler} and integrating by parts, 
\begin{equation} \begin{aligned} 
\int \bar u \cdot \p_t u = - \int \bar u _k \bar u _l u_{k,l} + \bar u_k \bar u_{l,k} u_l 
%& =  \int ( \bar u _k \bar u _l )_ {,l} u_k + ( \bar u _k u_l ) _{,k} \bar u _l \\
 = \int \bar u_{k,l} \bar u _{l} u_k - \bar u_k \bar u_{l,k} u_l = 0 %+ \bar u_k u_{l,k} \bar u _l \, .
\end{aligned} \end{equation} 
after swapping the indices $k$ and $l$ in the last term. \dacomment{Simplified a step} %roles of the
%\begin{equation} 
%  \int \bar u _k \bar u _l u_{k,l} + \bar u_k \bar u_{l,k} u_l =0,
%  \end{equation} 
%  so energy is conserved.
\end{comment}

The complexification preserves Galilean invariance: If $u$ is a solution, then for any $c \in \R^d$, $\tilde u ( x,t ) = u ( x-c, t ) + c$ is also a solution. Furthermore, $\Re \int u$ is constant in time:
%Integrating by parts, we also see 
\begin{equation}
\label{eq:zerothmode}\begin{aligned}
\frac{d}{dt} \int u_k & = - \int \bar u _l u_{k,l} + \bar u_{l,k} u _l = - \int \bar u _{l,k} u_l = \Im \int \bar{u}_l u_{l,k} \, .
%& = - \int - \bar u_{l,l} u_k +  =\\
%& = - \int \frac 12 ( \bar u_{l,k} u_l - \bar u _l u_{l,k} ).
\end{aligned} \end{equation}
%The right-hand side is purely imaginary, the real part of the real part of the zeroth Fourier mode of   this fact 
%Combining this fact with Galilean invariance,
Hence, \emph{we assume without loss of generality below that} $\Re \int u = 0$. % is zero. 

In addition to conservation of energy, the model~\eqref{eq:ComplexEuler} keeps conservation of `enstrophy' $\int \omega^2 \, dx$, 
%\begin{equation} %\end{equation}
as is evident from the equation
\begin{equation}
    \p_t \omega^2 + \bar{u} \cdot \nabla \omega^2 = 0 \, .
\end{equation}
More generally, we have `Casimirs' $f(\omega)$, given by any complex-analytic function $f : \C \to \C$, since
\begin{equation}
    \p_t f(\omega) + \bar{u} \cdot \nabla f(\omega) = 0 \, .
\end{equation}
In three dimensions, the vorticity $\omega = \curl u = \nabla \times u$ satisfies the equation
\begin{equation}
\p_t \omega + [\bar{u},\omega] = 0 \, ,
\end{equation}
see~\cite[(3.31)]{VS}, 
and the helicity $\int u \cdot \omega \, dx$ 
%\begin{equation}\end{equation}
is conserved: We compute
\begin{equation}
    \p_t u \cdot \omega + (\bar{u} \cdot \nabla) u \cdot \omega + (\nabla \bar{u})^T u \cdot \omega + \nabla p \cdot \omega = 0
\end{equation}
\begin{equation}
    \p_t \omega \cdot u + (\bar{u} \cdot \nabla) \omega \cdot u - (\omega \cdot \nabla) \bar{u} \cdot u = 0 \, .
\end{equation}
Summing the two equations yields
\begin{equation}
    \p_t (u \cdot \omega) + \bar{u} \cdot \nabla(u \cdot \omega) + \nabla p \cdot \omega = 0 \, .
\end{equation}
%(The additional term in the velocity equation cancels the vortex stretching term.) Then integrate.

We conclude the section with a statement of analytic solvability for the model \!\eqref{eq:ComplexEuler}. 

Consider the Wiener algebra $A^r$ consisting of distributions $u$ on the torus whose Fourier coefficients satisfy
\begin{equation}
    \| u \|_{A^r} := \sum e^{r|k|} |\hat{u}(k)| < +\infty \, .
\end{equation}
We do not distinguish notation between scalar- and vector-valued function spaces except when necessary. $A^r$ is evidently a (non-unital) Banach algebra under pointwise multiplication:
\begin{equation}
    \| u v \|_{A^r} \leq \| u \|_{A^r} \| v \|_{A^r} \, ,
\end{equation}
whenever $r \geq 0$ and $u, v \in A^r$. For $r > 0$, $A^r$ is a space of analytic functions; $r$ is the `analyticity radius'.
\begin{proposition}[Analytic well-posedness]
Assume $u^{\rm in} \in A^{r_0} $. Then there exists $\eta >0$ and a unique solution $u(\cdot, t )$ such that for every $r \in (0, r_0),$ and for $t$ with $|t|< \eta ( r_0 - r ), $ $u(\cdot,t)$ is continuously differentiable in $t$ with values in $A^r$. 
\end{proposition} 

\begin{proof} 
We apply Nirenberg's abstract Cauchy-Kovalevskaya theorem (\cite{Nirenberg}, Theorem 1.1). We verify the hypotheses of the theorem. Assume $r > r^\prime.$ First, we have $A^r \subset A^{r^\prime} $ and $\| u \| _{ A^{r^\prime} } \leq \| u \| _{ A^r}$. 
We have the estimate 
\begin{align}  \label{eq:gradientestimate}
\| \nabla u \| _{A^{r^\prime}}&  = \sum |k| e^{r^\prime | k | } | \hat u (k) | 
 = \sum |k|e^{(r^\prime - r ) | k| } e^{r | k | }  | \hat u (k) | \nonumber\\
& \leq \sup (e^{(r^\prime - r ) | k| }|k|) \| u \| _{A^r}
 \leq \frac{ e}{r - r^\prime}\| u \| _{A^r}. 
 \end{align}
Let $F(u) = - \mathbb{P} ( \bar u \cdot \nabla u + ( \nabla \bar u )^T u ).$
Assume $\|u \|_ {A^r} , \| v \| _{ A^r } < M $. Then, 
\begin{align} 
\| F (u ) \!-\! F (v) \|_{A^{r^\prime}} 
& 
\leq \| \bar u \cdot \nabla ( u\!-\!v ) \!+\! ( \bar u \!-\! \bar v ) \cdot \nabla v \!+\! ( \nabla \bar u ) ^\top ( u\!-\!v) \!+\! ( \nabla ( \bar u \!-\! \bar v ) )^\top v \| _{A^{r^\prime }}  \nonumber\\
& \leq \| u \|_{A^{r^\prime }} \| \nabla ( u \!-\! v ) \|_{A^{r^\prime }}
\!+\! \| \nabla v  \|_{A^{r^\prime }} \|  u \!-\! v  \|_{A^{r^\prime }}
\!+\! \|\nabla  u \|_{A^{r^\prime }} \|  u \!-\! v  \|_{A^{r^\prime }}
\nonumber\\ & \quad + \| v \|_{A^{r^\prime }} \| \nabla ( u - v ) \|_{A^{r^\prime }} \nonumber\\
& \leq \frac{ CM}{r - r^\prime } \| u - v \| _{ A^r }.
 \end{align} 
Therefore $F$ maps $A^r$ continuously into $A^{r^\prime}$. 
Let $A_u (v) = - \mathbb{P} ( \bar u \cdot\nabla  v + \bar v \cdot \nabla u + ( \nabla \bar u )^T v + (\nabla  \bar v)^T u ) $. 
We have 
\begin{align} 
\| F ( u ) - F(v) - A_{u}(u-v) \| _{ A^{r^\prime}} & \leq \| ( \bar u - \bar v ) \cdot \nabla ( v - u ) + ( \nabla ( \bar u - \bar v )^T ) ( v -u ) \| _{ A ^{r ^ \prime }} \nonumber\\
& \leq \frac{ C}{r - r^\prime} \| u - v \|^2. 
 \end{align} 
\end{proof}

\section{Complex shear flows}

In this section, we prove ill-posedness (Theorem~\ref{thm:ill-posedness}) and finite-time loss of analyticity (Corollary~\ref{cor:finitetimeblowup}) for the model~\eqref{eq:ComplexEuler}.

We consider solutions which are translation invariant in $y$, namely, $u = u(x,t)$. This symmetry is preserved under the evolution. In particular, the condition $\div u = 0$ yields that $u = (u_1,u_2)$ satisfies $u_1 = a_1(t)$ is constant-in-$(x,y)$. It will be convenient to decompose $u = a(t) + b(x,t) e_2$, where $\int_\T b \, dx = 0$. That is, $u$ is decomposed as a mean-zero shear flow $b$ and a constant background flow $a$, which is purely imaginary. 

The Euler equations become
\begin{equation}
\left\lbrace
\begin{aligned}
    \p_t a_1 + \fint {\p_x \bar b} b \, dx &= 0 \\
    \p_t a_2 &= 0 \\
    \p_t b  + \bar{a_1} \p_x b &= 0 \, .
    \end{aligned} \right.
\end{equation}
Subsequently, we may assume that $a_2 = 0$ without loss of generality. Next, we write
\begin{equation}
    b = \sum_{k \neq 0} b_k e^{ikx} \, , \quad k \in  \Z \, .
\end{equation}
This diagonalizes the $b$ equation. In the $a_1$ equation,
\begin{equation}
\fint {\p_x \bar b} b \, dx = - i \sum_{k \neq 0} k |b_k|^2 \, .
\end{equation}
It will be further convenient to write $a_1 = iq$, $q \in \R$.

In conclusion, the resulting infinite-dimensional ODE system is
\begin{equation}
    \label{eq:ODEsystem}
    \left\lbrace
\begin{aligned}
    \p_t q &= \sum_{k \neq 0} k |b_k|^2 \\
    \p_t b_k &= - q k b_k \, ,
    \end{aligned} \right.
\end{equation}
supplemented with the initial condition $(q,b)|_{t=0} = (q^{\rm in},b^{\rm in})$.

We say that $(q,b)$ is a \emph{weak solution} on the finite open interval $I$ if $b \in L^2_t H^{1/2}_x(\T \times I)$, which, in particular, guarantees that $\p_t q \in L^1(I)$, and the ODEs~\eqref{eq:ODEsystem} are satisfied in the integral sense. It follows that $q \in C(\bar{I})$ and $\p_t b \in L^2_x H^{-1/2}_x$. This is enough to show that weak solutions conserve the energy $|q|^2 + \sum_k |b_k|^2$. %, as is easily seen from the representation formula.

The dynamics of~\eqref{eq:ODEsystem} are not difficult to understand. For example, consider the case where $b_k^{\rm in}=0$ for all but one value of $k$ and $b_k \in \mathbb{R}$. The first equation of the system \eqref{eq:ODEsystem} simplifies to 
\begin{equation}
	\label{eq:simplifiedkthing}
    \p_t q = k b_k^2 \, .
\end{equation}
Since the energy is conserved, we introduce a new parameter $\theta$ and write 
\begin{equation}
q = -\sqrt E \cos \theta\, , \quad b_k = \sqrt E \sin \theta\, , \quad E = q^2 + b_k ^2\,  , \quad \theta \in (-\pi, \pi ] \, .
\end{equation}
We then have $\partial_t q = \sqrt E \sin \theta \, \p_t \theta$. At the same time,~\eqref{eq:simplifiedkthing} gives $\p_t q = k E \sin ^2 \theta $, so
\begin{equation}\label{eq:simplifiedequation}
\p_t \theta = \sqrt E k \sin \theta \, .
\end{equation}
This system has two fixed points, $\theta=0$ (unstable) and $\theta=\pi$ (stable). If $\theta^{\rm in} \neq  0$, then the solution evolves to $\theta= \pm \pi$ as $t \to + \infty.$ In particular, $\sin \theta = \pm 1$ after finite time, when all the energy of the system becomes momentarily concentrated in the $k^{\rm th}$ Fourier mode.~\eqref{eq:simplifiedequation} is also readily seen to be a time rescaling of $\p_t \theta = \sin \theta$. %so

The solutions above are enough to prove ill-posedness. For simplicity, we present only the Sobolev ill-posedness; the analogous proof works in Gevrey spaces below analytic regularity, also in Corollary~\ref{cor:illposednessfullynonlinear}.

\begin{theorem}[Ill-posedness]
\label{thm:ill-posedness}
Fix $\varepsilon> 0$, $T>0$, and $M>0$. Then there exists $u^{\rm in}$ with $\| u^{\rm in}\|_{H^s} < \varepsilon$ and $\sup_{[0,T]} \| u ( \cdot, t ) \| _{H^s} > M$.
\end{theorem}

\begin{proof} Let $q^{\rm in } =- \frac \varepsilon 2 , $ $b^{\rm in} = \frac \varepsilon {2 \langle k \rangle ^s } e^{ik x } $ where $k$ is to be determined. Evolving according to \eqref{eq:ODEsystem}, there is some $T_0>0$ such that $\| u( \cdot , T_0 ) \| _{H^s} \geq \frac{ \varepsilon}{4} \langle k \rangle ^s. $ Assume $k$ is large enough that $\frac{ \varepsilon}{4} \langle k \rangle ^s> M$. The result will follow by estimating $T_0$. Since $\sin \theta \geq \frac{ 2 }{ \pi } \theta $ for $\theta \in [0 , \frac \pi 2 ],$ we have 
\begin{equation}
\partial_t \theta \geq \frac \varepsilon 2 k \sin \theta \geq \frac \varepsilon \pi k \theta
\end{equation} for $\theta \leq \frac \pi 2. $ Therefore $\theta \geq \theta^{\rm in} e^{\frac{ \varepsilon}{\pi} k t }  \approx \langle k \rangle^{-s} e^{\frac{ \varepsilon}{\pi} k t } $. We can take $T_0$ to be the time at which $\theta =1$, so 
\begin{equation}
T_0 \lesssim \frac{\pi}{k \varepsilon } \log \langle k \rangle^s \, .
\end{equation}
Since $T_0 \to 0$ as $k \to +\infty$, we may choose $k$ large enough so that $T_0 < T.$
\end{proof}

We now demonstrate the finite-time loss of analyticity.
\begin{lemma}
    \label{lem:localwellposedness}
    Suppose that $q^{\rm in} > 0$, $b^{\rm in} \in L^2$, and $b_k = 0$ for $k \leq 0$. Then there exists a unique global-in-time weak solution $(q,b) \in C([0,+\infty)) \times L^2_t H^{1/2}_x(\T \times \R_+)$ to~\eqref{eq:ODEsystem} with initial data $(q^{\rm in},b^{\rm in})$.
\end{lemma}
\begin{proof}
    First, we prove the local theory. Write $q = q^{\rm in} + \tilde{q}$. The integral formulation of the equation is
    \begin{align}
      &  \tilde{q}(t) = \int_0^{t} \sum_{k \geq 0} k|b_k|^2 \, ds
\\
     &   b_k(t) = e^{-q^{\rm in} k t} b_k^{\rm in} -  \int_0^t e^{-q^{\rm in} k (t-s)} \tilde{q} b_k(s) \, ds \, .
    \end{align}
    %(Not strictly necessary to write $b$ equation in integral form.)
    By elementary energy estimates, we have that
    \begin{equation}
        \label{eq:initialdataestimate}
        q^{\rm in} \sum_{k \geq 0} k |e^{-q^{\rm in} k t} b_k^{\rm in}|^2 \leq \frac{1}{2} \| b^{\rm in} \|_{L^2}^2 \, ,
    \end{equation}
    since $e^{-q^{\rm in} k t} b_k^{\rm in}$ is the Fourier representation of the solution of the PDE $\p_t c = - q^{\rm in} |\nabla| c$ with $c|_{t=0} = b^{\rm in}$.
    The left-hand side of~\eqref{eq:initialdataestimate} will be small provided that time is taken sufficiently small. It remains to demonstrate that the bilinear forms
    \begin{align}
        \label{eq:firsbilinearform}
   & (b,c) \mapsto \int_0^{t} \sum_{k \geq 0} kb_k \bar{c_k} \, ds
\\
        \label{eq:secondbilinearform}
  &  (\tilde{q},c) \mapsto - \sum_{k \neq 0}  \int_0^t e^{-q^{\rm in} k (t-s)} \tilde{q} k c_k(s) \, ds \,  e^{ikx}
    \end{align}
    are bounded $L^2_t H^{1/2}_x(\T \times (0,\varepsilon)) \times L^2_t H^{1/2}_x(\T \times (0,\varepsilon)) \to C([0,\varepsilon])$ and $L^2_t H^{1/2}_x(\T \times (0,\varepsilon)) \times C([0,\varepsilon]) \to L^2_t H^{1/2}_x(\T \times (0,\varepsilon))$, respectively, with constants uniform in $\varepsilon$ small. The estimate on~\eqref{eq:firsbilinearform} follows from Cauchy-Schwarz and the definition of the $L^2_t H^{1/2}_x$ norm. The second bilinear form~\eqref{eq:secondbilinearform} is the Fourier representation of the solution to the PDE
    \begin{equation}
        \p_t b = - q^{\rm in} |\nabla| b - \tilde{q} |\nabla| c
    \end{equation}
    with $b|_{t = 0} = 0$, and the required estimate is an easy energy estimate. With these ingredients, local-in-time existence and uniqueness for~\eqref{eq:ODEsystem} follow from the abstract Picard lemma (see, e.g., Lemma A.1 in~\cite{gallagherasymptotics} or Lemma 5.5, p. 217, in~\cite{bahouri}), as is typical in the Navier-Stokes well-posedness theory.

    Second, we address the global theory. As easily seen from the representation formula $b_k = e^{-k \int_0^t q(s) \, ds} b_k^{\rm in}$, the solution is immediately analytic and can be continued provided that it remains analytic. Since $q$ is increasing and, in particular, $q$ is bounded below by $q^{\rm in}$ and above by the (conserved) total $L^2$ norm, its radius of analyticity grows linearly as $t \to +\infty$. Hence, the solution is global in time.
\end{proof}

\begin{remark}
    More generally, the system~\eqref{eq:ODEsystem} is locally well-posed forward-in-time under the assumption that $q^{\rm in} > 0$, $b^{\rm in} \in L^2$, and $P_{k \leq 0} b^{\rm in}$ is analytic.
\end{remark}

\begin{corollary}[Finite-time loss of analyticity]
\label{cor:finitetimeblowup}
There exist analytic solutions $(q,b)$ on $(-\infty,0)$ satisfying $b|_{t=0} \in L^2 \setminus \cup_{s > 0} H^s$ and which cannot be extended to a weak solution on $(-\infty,\varepsilon)$ for any $\varepsilon > 0$.
\end{corollary}
\begin{proof} Reverse time in the solutions of Lemma~\ref{lem:localwellposedness} with $b^{\rm in} \in L^2 \setminus \cup_{s > 0} H^s$. (For example, use that if $u(x,t)$ solves \eqref{eq:ComplexEuler}, then $u(-x,-t)$ is also a solution. The resulting solutions are supported in non-positive Fourier modes.) If one of these solutions could be extended as a weak solution, then by continuity, we would have $q > 0$ in a neighborhood of the origin. Since weak solutions have finite energy, then, by the representation formula, $e^{\varepsilon |\nabla|} b^{\rm in} \in L^2$ for sufficiently small $\varepsilon$. In particular, $b^{\rm in}$ would be analytic, a contradiction.
\end{proof}

%\dacomment{would be better to use actual time inversion}

\begin{remark}
    The finite-time blow-up holds even with fractional dissipation $- \nu |\nabla|^\alpha$ on the right-hand side of the equation when $\alpha \in (0,1)$ or when $\alpha = 1$ and $\nu < q^{\rm in}$.
\end{remark}

%\da{
%Unstable manifold:

%For every ..., there exists a solution which...

%For every solution converging to..., there exists...
%}

%Also, blow-up in infinite time...?

% Can't find solutions which instantly leave, it seems, 'cuz can't make sense of equation anymore

\begin{remark}
We compare with the na{\"i}vely complexified Euler equations
\begin{equation}
\label{eq:naivecomplexEuler}
    \p_t u + u \cdot \nabla u + \nabla p = 0 \, , \quad \div u = 0 \, ,
\end{equation}
where $u$ is a complex vector field. We have conservation for $\int u^2$ and $\int \omega^2$. However, unlike the complexified Euler equations,~\eqref{eq:naivecomplexEuler} preserves the mean flow. Shear flow solutions of the form $(0,b(x))$ are steady states, whereas solutions of the form $(i,b(x))$ satisfy
\begin{equation}
    \p_t b + i \p_x b = 0 \, ,
\end{equation}
which is evidently ill-posed and exhibits finite-time loss of analyticity. Surprisingly,~\eqref{eq:naivecomplexEuler} already arose in~\cite{Caflisch1993Complex} in the context of axisymmetric vortex sheets.
\end{remark}

\section{Ill-posedness for non-hyperbolic systems}
\label{sec:nonhyperbolicsystems}

We now consider
\begin{equation}
    \label{eq:generallocalPDE}
    \p_t u + F(u,\p_x u) = 0
\end{equation}
where $u : \T \times I \to \R^m$ and $I \subset \R$ is a time interval. The nonlinearity $F : \mathbb{R}^m \times \mathbb{R}^m \to \mathbb{R}^m$ is assumed to satisfy

\begin{quote}
\textbf{(A1)}
    There exists $c \in \mathbb{R}^m$ such that $F(c,0) = 0$, 
    and $F$ is real analytic in a neighborhood of $(c,0)$.
    \end{quote}
    
 In particular,
\begin{equation}
    F = \p_u F(c,0) (u-c) + \p_p F(c,0) p - \tilde{F}(u-c,p) \, ,
\end{equation}
where $\tilde{F} = O(|u-c|^2 + |p|^2)$ is real analytic. The above assumptions can treat complex equations, e.g., the complex Burgers equation $\p_t u + u \p_x u = 0$, by embedding them into real systems.

We write $u = c + w$. Then $w$ solves the perturbed equations
\begin{equation}
    \label{eq:perturbedequation}
    \p_t w - Lw = \tilde{F}(w,\p_x w) \, ,
\end{equation}
and $L$ is the linearized operator
\begin{equation}
    - L w := \p_u F(c,0) w + \p_p F(c,0) \p_x w \, .
\end{equation}
Under the Fourier transform, we have
\begin{equation}
    L_k w_k := (L w)_k = L_0 w_k + k \bm{L} w_k \, ,
\end{equation}
where $L_0 := - \p_u F(c,0)$ and $\bm{L} := - i \p_p F(c,0)$ are complex $m \times m$ matrices.

We further assume
\begin{quote}
\textbf{(A2)} $\bm{L}$ is hyperbolic
\end{quote}
in the dynamical systems sense, namely, it has no spectrum on the imaginary axis.

%Finally, we suppose that $\bm{L}$ has unstable spectrum:
%\begin{quote}
 %   \textbf{(A3)} $\sigma(\bm{L}) \cap \{ \Re \lambda > 0 \}$ is non-empty.
%\end{quote}

The situation is less difficult when $L_0 = 0$ and $\sigma(\bm{L})$ is semi-simple, as in the complex Burgers equation~\eqref{eq:naivecomplexEuler}, but the assumptions $\textbf{(A1)}$-$\textbf{(A2)}$ seem reasonable, and we wish to treat them. In this general setting, we require some perturbation theory:

For $k \neq 0$, we have $\sigma(L_k) = k \sigma(\bm{L} + k^{-1} L_0)$. By standard finite-dimensional perturbation theory~\cite{Katobook}, the spectrum of $\sigma(\bm{L} + \varepsilon L_0)$ is `continuous' in the sense that, although Jordan blocks associated to $\bm{L}$ might split in a non-smooth way, the eigenvalues of $\bm{L} + \varepsilon L_0$ converge as $\varepsilon \to 0$ to the eigenvalues of $\bm{L}$, and the spectral projections onto sufficiently small neighborhoods of the eigenvalues of $\bm{L}$ are smooth in $\varepsilon$ when $|\varepsilon| \ll 1$. For $\mu \in \sigma(\bm{L})$, we say that the eigenvalues of $\bm{L} + \varepsilon L_0$, $|\varepsilon| \ll 1$, converging to $\mu$ belong to the \emph{eigenvalue group} $\Lambda(\mu, \varepsilon)$ of $\mu$. We call the corresponding spectral subspace $E(\mu, \varepsilon)$ the \emph{group eigenspace} of $\mu$, see~\cite[p. 67-68]{Katobook}. One simple consequence is that $L_k$ is hyperbolic for sufficiently large $|k|$. Since $\sigma(L) = \cup_{k} \sigma(L_k)$, clearly, $\sigma(L)$, where $L$ is considered on $L^2(\T)$, is discrete and consists of eigenvalues; $L$ has at most finitely many central eigenvalues, and the remaining eigenvalues are bounded away from the imaginary axis. Given $\gamma \in \R$, we define $P^u_\gamma$ to be the spectral projection onto the unstable subspace $E^u_\gamma$ corresponding to (generalized) eigenvectors with eigenvalues satisfying $\Re \lambda > \gamma$. Similarly, $P^{cs}_\gamma$ is the spectral projection onto the center-stable subspace $E^{cs}$ (eigenvalues with $\Re \lambda \leq \gamma$). These subspaces and projections are well defined, even at the level of distributions: first project onto Fourier modes and subsequently onto unstable and center-stable subspaces. Projections onto the group eigenspaces, which are well defined when $|k| \gg 1$, are uniformly bounded.

\emph{Functional set-up}. For $r,s \geq 0$, consider $A^{r,s} := \{ f \in A^r : |\nabla|^s f \in A^r \}$ with the norm $\| f \|_{A^{r,s}} = \| \langle \nabla \rangle^s f \|_{A^r}$. %We further conisder subspace of $C((-\infty,0];A^{r,1})$ with norm
%\begin{equation}
 %   \sum_k \sup_{t \in (-\infty,0]} e^{\gamma |t|+r |k|} |k|  |w_k(t)| \, \quad \left( \geq \sup_{t \in (-\infty,0]} e^{\gamma |t|} \| w \|_{A^{r,1}} \right) \, ,
%\end{equation}
We introduce the notation
\begin{equation}
    \| w \|_{\tilde{L}^\infty_\alpha(I;A^{r,s})} := \sum_k \esssup_{t \in I} \langle k \rangle^s e^{- \alpha t+r|k|} |w_k| \, ,
\end{equation}
where we emphasize that the supremum in $t$ is inside the summation in $k$. It is furthermore admissible to prescribe $r = r(t)$ a function of time, e.g., $r = \nu|t|$. We use the notation $\tilde{C}_\alpha(I;A^{r,s})$ to indicate that additionally the Fourier modes are continuous, as is automatic for solutions of the equation. $\alpha > 0$ indicates decay backward-in-time at a rate $\alpha$. 

The constants below may depend on $m, L_0, \bm{L}, F$, etc.

\begin{theorem}[Unstable manifold]
    \label{thm:unstablemanifold}
    If $0 < \gamma \leq \sfrac{m_0}{2} :=\sfrac12 \min [\Re (\sigma(L)) \cap (0,+\infty)]$, then under the assumptions \textbf{(A1)}-\textbf{(A2)}, there exist $\varepsilon_0 > 0$ and a one-to-one analytic map $\overline{B^{A^{0,1}}_{\varepsilon_0}} \cap E^u \to \tilde{C}_\gamma(\R_-;A^{0,1})$ which sends sufficiently small data $a_0 \in A^{0,1}$ in the unstable subspace $E^u$ to solutions $w$ of the PDE \eqref{eq:perturbedequation} with $P^u w(0) = a_0$. There exists $\zeta > 0$ such that solutions obey the estimates
    \begin{equation}
        \label{eq:estimates}
        \| w \|_{\tilde{L}^\infty_\gamma(\R_-;A^{\zeta |t|,1})} \lesssim \| a_0 \|_{A^{0,1}} \, .
    \end{equation}

        Conversely, every solution in $\tilde{L}^\infty_\kappa((-\infty,T);A^{0,1})$ for some $T \in \R$ and $\kappa > 0$ belongs, after a time translation, to the image of this map.

        There exists $\varepsilon_1 > 0$ such that, if $\gamma \geq \sfrac{m_0}{2}$ and $b_0 \in E^{cu}_\gamma \cap \overline{B^{A^{0,1}}_{\varepsilon_1}}$, then there exists a unique solution $w \in \tilde{C}_{\sfrac{3\gamma}{2}}(\R_-;A^{0,1})$ satisfying the asymptotics
    \begin{equation}
        \label{eq:asymptotics}
        \| w - e^{\cdot L} b_0 \|_{\tilde{L}^\infty_{\sfrac{3\gamma}{2}}(\R_-;A^{0,1})} \lesssim \| b_0 \|_{A^{0,1}}^2 \, .
    \end{equation}
\end{theorem}

% look at statements of standard unstable manifold thms to get a feel for what should be written.
%As a consequence, 
We obtain finite-time loss of analyticity from the backward-in-time smoothing estimate~\eqref{eq:estimates} when $a_0 \in A^{0,1} \setminus \cup_{s>1} A^{0,s}$. For ill-posedness below analytic regularity,~\eqref{eq:asymptotics} is more convenient, see Corollary~\ref{cor:illposednessfullynonlinear}.

$a_0$ is (the unstable projection of the) initial data. $b_0$ is scattering data. By the estimate~\eqref{eq:asymptotics}, small $b_0$ beget solutions with small $a_0$ and, therefore, satisfy~\eqref{eq:estimates}. One could also study the map $a_0 \mapsto b_0$.

\emph{Representation formula}. Suppose that $w \in L^\infty_t {\rm Lip}_x(\T \times (-\infty,T))$ is a solution to~\eqref{eq:perturbedequation} satisfying
\begin{equation}
    \label{eq:exponentiallydecaying}
    |w|,|\p_x w| = O(e^{\varepsilon t}) \text{ as } t \to -\infty
\end{equation}
for some $\varepsilon > 0$.
We write Duhamel's formula
\begin{equation}
    w(t_1) = e^{(t_1-t_0) L} w(t_0) + \int_{t_0}^{t_1} e^{(t_1-s)L} \tilde{F}(w,\p_x w)(s) \, ds
\end{equation}
which is unequivocally defined mode-by-mode in $x$, for $t_0 < t_1 < T$, and we suppress the $x$-dependence of functions in the notation. First, we project onto the center-stable subspace $E^{cs}$:
\begin{equation}
    P^{cs} w(t_1) = e^{(t_1-t_0) L} P^{cs} w(t_0) + \int_{t_0}^{t_1} e^{(t_1-s)L} P^{cs} \tilde{F}(w,\p_x w)(s) \, ds \, .
\end{equation}
By the assumption~\eqref{eq:exponentiallydecaying}, and because we are applying the semigroup forward-in-time, we obtain the following formula when $t_0 \to -\infty$ and $t_1 = t$:
\begin{equation}
    \label{eq:Duhamelforward}
    P^{cs} w(t) = \int_{-\infty}^{t} e^{(t-s)L} P^{cs} \tilde{F}(w,\p_x w)(s) \, ds \, .
\end{equation}
Meanwhile, we apply $e^{(t_0-t_1) L} P^u$ to~\eqref{eq:perturbedequation}, write $t_0 = t$, and rearrange:
\begin{equation}
    \label{eq:Duhamelbackward}
    P^u w(t)  =  e^{(t-t_1)L } P^u w(t_1) - \int_t^{t_1} e^{(t-s)L} P^u \tilde{F}(w,\p_x w)(s) \,ds \, .
\end{equation}
%(This is Duhamel's formula for evolving backward.)
We sum~\eqref{eq:Duhamelforward} and~\eqref{eq:Duhamelbackward} to obtain the representation formula
\begin{align}
    \label{eq:representationformula}
    w(t) &= e^{(t-t_1)L } P^u w(t_1) \nonumber\\
    &\quad - \int_t^{t_1} e^{(t-s)L} P^u \tilde{F}(w,\p_x w)(s) \,ds \!+\! \int_{-\infty}^{t} e^{(t-s)L} P^{cs} \tilde{F}(w,\p_x w)(s) \, ds \, .
\end{align}

We now refine our functional set-up. We have the algebra property
\begin{equation}
    \| fg \|_{\tilde{L}^\infty_\gamma(I;A^r)} \leq \| f \|_{\tilde{L}^\infty_{\gamma_1}(I;A^r)} \| g \|_{\tilde{L}^\infty_{\gamma_2}(I;A^r)} \, ,
\end{equation}
where $\gamma \leq \gamma_1 + \gamma_2$ and $r = r(t)$ is possibly time-dependent.
Therefore, we have the following composition estimate
\begin{equation}
	\label{eq:Fcompositionestimate}
\begin{aligned}
    \| F(u,p) \|_{\tilde{L}^\infty_\gamma(I;A^r)}  \leq \sum_{\alpha,\beta} \frac{|\p^\alpha_u \p^\beta_p F(0,0)|}{\alpha!\beta!} \| u \|_{\tilde{L}^\infty_\gamma(I;A^r)}^{|\alpha|}   \| p \|_{\tilde{L}^\infty_\gamma(I;A^r)}^{|\beta|} \, ,
    %=: \mathbf{F}(\|u\|_{\tilde{L}^\infty_\gamma(I;A^r)} , \|p\|_{\tilde{L}^\infty_\gamma(I;A^r)} ) \, ,
    \end{aligned}
\end{equation}
where $\alpha,\beta$ are multi-indices.

%Suppose we have the following growth estimates:
%\begin{equation}
    %|e^{L_k t} (P^u w)_k| \leq C e^{\gamma t} \, , \quad t \leq 0 \, ,
%\end{equation}
%\begin{equation}
    %|e^{L_k t} (P^{cs} w)_k| \leq C e^{\gamma t} \, , \quad t \geq 0 \, .
%\end{equation}
%and $\delta = 0$ provided that the central eigenvalues are semi-simple.
\emph{Linear estimates}. We have the smoothing and decay estimates
\begin{align}
    \label{eq:backwardsmoothing}
    |e^{L_k t} (P^u w)_k| &\lesssim e^{(\gamma + \nu|k|)t} \, , \quad t \leq 0 \, , \\
    \label{eq:forwardsmoothing}
    |e^{L_k t} (P^{cs} w)_k| &\lesssim_\gamma e^{(\gamma - \nu |k|)t} \, , \quad t \geq 0 \, ,
\end{align}
where $0 < \gamma \leq \sfrac{3m_0}{4}$, and $\nu > 0$ is sufficiently small depending on $\gamma$. $\nu$ and the constant in~\eqref{eq:forwardsmoothing} are uniform when $\gamma$ is away from zero.

We may deduce mapping properties from exponential integrals:
%\begin{equation}
    \begin{align}
    \notag
    &e^{(b + \zeta |k|)|t|} \left| \int_{-\infty}^t e^{(d - \nu |k|)(t-s)} f_k(s) \, ds \right| \\
    \label{eq:forwardevolution}
    & \leq [(b-d) + (\nu - \zeta)|k|]^{-1} \| e^{(b + \zeta|k|) |\cdot|} f_k \|_{L^\infty(-\infty,t)} \, , \quad d < b \, , \\
    %\end{aligned}
%\end{equation}
%\begin{equation}
    %\begin{aligned}
    \notag
	&e^{(b + \zeta |k|)|t|} \left| \int_{t}^0 e^{(d + \nu |k|)(t-s)} f_k(s) \, ds \right| \\
	\label{eq:backwardevolution}
    &  \leq [(d-b) + (\nu - \zeta)|k|]^{-1} \|  e^{(b + \zeta|k|) |\cdot|} f_k \|_{L^\infty(t,0)} \, , \quad d > b \, ,
    \end{align}
%\end{equation}
\begin{comment}
\begin{equation}
    \label{eq:forwardevolution}
    \langle k \rangle e^{(\gamma + \zeta |k|)t} \left| \int_{-\infty}^t e^{(\gamma - \delta - \nu |k|)(t-s)} f_k(s) \, ds \right| \les_\delta (\nu-\zeta)^{-1} \|  e^{(\gamma + \zeta|k|) \cdot} f_k \|_{L^\infty(-\infty,t)}
\end{equation}
\begin{equation}
    \label{eq:backwardevolution}
    \langle k \rangle e^{(\gamma + \zeta |k|)t} \left| \int_{t}^0 e^{(\gamma + \delta + \nu |k|)(t-s)} f_k(s) \, ds \right| \les_\delta (\nu-\zeta)^{-1} \|  e^{(\gamma + \zeta|k|) \cdot} f_k \|_{L^\infty(t,0)} \, ,
\end{equation}
\end{comment}
where $t \leq 0$. Here, $\zeta \in [0,\sfrac{\nu}{2}]$ is the rate at which the analyticity radius grows, and we typically choose $\zeta = \sfrac{\nu}{2}$ ($\zeta = 0$ for the uniqueness statement).

%Need to be careful to spectral gaps. There may be arbitrarily small gaps at high frequencies if the unstable eigenvalues the ratios of unstable eigenvalues are not rational multiples of each other. \da{Think about how to set this bullshit up.}

%\begin{equation}
 %   |e^{L_k t} (P^{u,cs} w)_k| \lesssim e^{-\nu |k| |t|} |(P^{u,cs} w)_k| \, , \quad t \lessgtr 0 \, , \quad |t| \leq 1 \, .
%\end{equation}
%(Notably, the forward center-stable evolution may not have long-time smoothing without a growing prefactor; for example, the $k=1$ mode could have a zero eigenvalue due to $L_0 \neq 0$.)

Let $a_0 \in A^{0,1} \cap E^u$. We wish to solve
\begin{align}
w(t) = \Phi[w](t) &:= e^{ tL } a_0 - \int_t^{0} e^{(t-s)L} P^u_\gamma \tilde{F}(w,\p_x w)(s) \,ds \nonumber\\
&\quad + \int_{-\infty}^{t} e^{(t-s)L} P^{cs}_\gamma \tilde{F}(w,\p_x w)(s) \, ds
\end{align}
via contraction mapping argument.
Evidently, $\| e^{tL} a_0 \|_{\tilde{L}^\infty_\gamma(\R_-;A^{\nu|t|,1})} \lesssim \| a_0 \|_{A^{0,1}}$ due to~\eqref{eq:backwardsmoothing} when $0 < \gamma \leq \sfrac{m_0}{2}$.

\emph{$\Phi$ stabilizes a ball}. The estimates~\eqref{eq:forwardevolution} and~\eqref{eq:backwardevolution} with $b = \gamma$ and $d = (1 \mp \sfrac{9}{10}) \gamma$ yield
\begin{align}
  &  \| \int_{-\infty}^t e^{(t-s) L} P^{cs} \tilde{F}(w,\p_x w)(s) \, ds \|_{\tilde{L}^\infty_\gamma(\R_-;A^{\zeta|t|,1})} \les_\gamma \| w \|_{\tilde{L}^\infty_\gamma(\R_-;A^{\zeta|t|,1})}^2
\\
  &  \| \int_t^{0} e^{(t-s) L} P^u \tilde{F}(w,\p_x w)(s) \,ds \|_{\tilde{L}^\infty_\gamma(\R_-;A^{\zeta|t|,1})} \les \| w \|_{\tilde{L}^\infty_\gamma(\R_-;A^{\zeta|t|,1})}^2 \, ,
\end{align}
provided that $\| w \|_{\tilde{L}^\infty_\gamma(\R_-;A^{\zeta|t|,1})}$ is less than, say, $\sfrac{1}{4}$ the radius of analyticity of $\tilde{F}$ (see the composition estimate~\eqref{eq:Fcompositionestimate}).
%Hence, $\| \Phi \|_{\tilde{L}^\infty_\gamma(\R_-;A^{r,1})} \leq C_{\rm stab} (\varepsilon + M^2)$.

\emph{$\Phi$ is contractive}. For the contraction estimate, we subtract $\Phi(w) - \Phi(v)$ and estimate similarly:
\begin{align}
    &\| \int_{-\infty}^t e^{(t-s) L} P^{cs} (\tilde{F}(w,\p_x w) - \tilde{F}(v,\p_x v))(s) \, ds \|_{\tilde{L}^\infty_\gamma(\R_-;A^{\zeta|t|,1})} \nonumber\\
    & \lesssim_\gamma \| w \|_{\tilde{L}^\infty_\gamma(\R_-;A^{\zeta|t|,1})}  \| w - v \|_{\tilde{L}^\infty_\gamma(\R_-;A^{\zeta|t|,1})} 
\\
    &\| \int_t^{0} e^{(t-s) L} P^u (\tilde{F}(w,\p_x w) - \tilde{F}(v,\p_x v))(s) \,ds \|_{\tilde{L}^\infty_\gamma(\R_-;A^{\zeta|t|,1})} \nonumber\\
    & \lesssim \| w \|_{\tilde{L}^\infty_\gamma(\R_-;A^{\zeta|t|,1})}  \| w - v \|_{\tilde{L}^\infty_\gamma(\R_-;A^{\zeta|t|,1})} \, ,
\end{align}
where we apply a composition estimate similar to~\eqref{eq:Fcompositionestimate} but Taylor expanding $\tilde{F}$ around $(w,\p_x w)$ instead of the origin.

In conclusion, the contraction mapping theorem produces a unique small solution which, in particular, satisfies the estimates~\eqref{eq:estimates} when $\gamma=\sfrac{m_0}{2}$. Uniqueness for general exponentially decaying solutions follows when we choose $\zeta = 0$ and $\gamma$ small. % (in which case the ball is chosen small depending on $\gamma$).

\begin{remark}[Analytic dependence]
A variation on the above argument will produce Lipschitz dependence on $a_0$. However, one could apply the implicit function theorem with analytic parameter dependence, see~\cite[p. 5-6]{kielhofer2011bifurcation}. We seek solutions to $\bm{F} = 0$, where
 \begin{align}
 	\label{eq:Fanalytic}
 	&\bm{F} : (a_0,w) \mapsto - w + e^{ tL } a_0 - \int_t^{0} e^{(t-s)L} P^u_\gamma \tilde{F}(w,\p_x w)(s) \, ds  \nonumber\\
 	&\quad\quad\quad\quad\quad\quad+ \int_{-\infty}^{t} e^{(t-s)L} P^{cs}_\gamma \tilde{F}(w,\p_x w)(s) \, ds \, .
 \end{align}
 Since trivially $\p_w \bm{F} = - {\rm Id}$, the goal is to prove that $\bm{F}$ is jointly analytic in a small ball in $(A^{0,1} \cap E^u) \times \tilde{L}^\infty_\gamma(\R_-;A^{\zeta |t|,1})$. For this, one can justify the following power series representation in $\tilde{L}^\infty_\gamma(\R_-;A^{\zeta |t|,1})$: 
 \begin{align}
&\int_t^{0} e^{(t-s)L} P^u_\gamma \tilde{F}(w,\p_x w)(s) \, ds \nonumber\\
 & = \sum_{\alpha,\beta} \frac{1}{\alpha! \beta!} \int_t^{0} e^{(t-s)L} P^u_\gamma [(\p_u^\alpha \p^\beta_p \tilde{F})(0,0) w^\alpha (\p_x w)^\beta](s) \, ds \, ,
 \end{align}
 and similarly for the center-stable evolution. We omit the remaining details.
\end{remark}

\textit{Leading order asymptotics}. We begin with comments on the projections $P^u_\gamma$ and $P^{cs}_\gamma$. Let $\mu \in \sigma(\bm{L})$. If $\mu$ is associated to a Jordan block which splits under perturbation by $\varepsilon L_0$, then there is no guarantee that $P^u_\gamma$ and $P^{cs}_\gamma$ are well behaved, i.e., uniformly bounded, when $\gamma = k\mu$, $|k| \gg 1$. Therefore, we define a modified spectral projection $\tilde{P}_\gamma$, which does not split the eigenvalue groups, in the following way. Let $\Pi_k$ be the projection onto the $k^{\rm th}$ Fourier mode (which we identify with its Fourier coefficient when convenient). %Let $M = \max \Re \sigma{\bm{L}}$ and $m = \min \Re \sigma{\bm{L}}$.\dacomment{Do we need $m$ and $M$?}
We consider $k_0$ sufficiently large such that the concept of eigenvalue group of $\bm{L} + k^{-1} L_0$ associated to each $\mu \in \sigma(\bm{L})$ is well defined when $|k| \geq k_0$. It is only necessary to explain how the projection acts on each group eigenspace $E(\mu,k^{-1})$, since the projections onto Fourier modes and group eigenspaces are uniformly bounded. Suppose that $\lambda \in k \Lambda  ( \mu , k^{-1})$ is an eigenvalue. When $|\lambda| \gg 1$, it is necessarily of the form $\lambda = k (\mu + o_{|k| \to +\infty}(1))$, where $|k| \geq k_0$ after possibly increasing $k_0$. Suppose that $\Re \lambda > \gamma$ and there exists $\lambda' = k (\mu + o_{|k| \to +\infty}(1)) \in k\Lambda(\mu, k^{-1})$ but with $\Re \lambda' < \gamma$. That is, suppose that the eigenvalue group $k \Lambda(\mu,k^{-1})$ is split across the line $\Re \lambda = \gamma$. In this case, we say that $\tilde{P}^u_\gamma \Pi_k \varphi =0 $ whenever $\Pi_k \varphi \in E ( \mu, k^{-1} ) $. % \jocomment{I changed this line. Does it make sense now?} 
Necessarily, the above situation occurs only when $\gamma = k (\mu + o(1))$ for some large $k$. In particular, $\tilde{P}^u_\gamma$ is a spectral projection onto certain eigenspaces corresponding to eigenvalues with real part \emph{at least} $\gamma$, and $\tilde{P}^{cs}_\gamma = I - \tilde{P}^u_\gamma$ is a spectral projection onto certain eigenspaces corresponding to eigenvalues with real part \emph{at most} $(1+o(1)) \gamma$, where the $o(1)$ factor is positive and, say, $\leq 1/200$. Otherwise, when $\gamma$ is not large, we simply define $\tilde{P}^u_\gamma = P^u_\gamma$ and $\tilde{P}^{cs}_\gamma = I - \tilde{P}^u_\gamma$. %, etc.

%$j\mu_i + o(1) \leq \theta \lambda$, then

%$j\mu_i + o(1))

%However, the spectral projections are well behaved away from this scenario. To this end, choose $\theta \in [3/4,1)$ satisfying the property that $\theta \Re \sigma(\bm{L}) \cap \Re \sigma(\bm{L}) = \emptyset$. This choice, combined with the previous finite-dimensional perturbation theory, yields that the projections $P^u_{\theta}$, $P^{cs}_\theta$, are uniformly bounded. If $\lambda \in \sigma L$ with $|\lambda| \gg 1$, then necessarily $\lambda = k (\mu + o(1))$ for $\mu \in \sigma{\bm{L}}$ and $|k| \gg 1$. Suppose $\tilde{\lambda} \in \sigma L$ satisfies $|\lambda - \tilde{\lambda}| \leq \delta |k|$. Then $|k (\mu+o(1)) - \tilde{k}(\tilde{\mu}+o(1))| \leq \delta |k|$. Then $\mu+o(1) - |k|^{-1}\tilde{k} \tilde(\mu+o(1)) \leq \delta$.

Let $\gamma \geq \sfrac{m_0}{2}$ and suppose that $b_0 \in E^{cu}_\gamma$. In particular, by the assumptions,
\begin{equation}
\| e^{tL} b_0 \|_{\tilde{L}^\infty_{\sfrac{9\gamma}{10}}(\R_-;A^{0,1})} \les \| a_0 \|_{A^{0,1}} \, .
\end{equation}

We have the new decay and growth estimates
\begin{align}
    \label{eq:decayestimatetilde}
   & |e^{L_k t} (\tilde{P}^u_{(\sfrac{3}{2} + 3 \delta) \gamma} w)_k| \lesssim e^{(\sfrac{3}{2} + 2 \delta) \gamma t} |(\tilde{P}^u_{(\sfrac{3}{2} + 3 \delta) \gamma } w)_k| \, , \quad t \leq 0 \, ,
\\
    \label{eq:intermediatemodes}
   & |e^{L_k t} ((\tilde{P}^{cs}_{(\sfrac{3}{2} + 3 \delta) \gamma } - P^{cs} w)_k| \lesssim e^{(\sfrac{3}{2} + 4 \delta)  \gamma t} | ((\tilde{P}^{cs}_{(\sfrac{3}{2} + 3 \delta) \gamma } - P^{cs}) w)_k| \, , \quad t \geq 0 \, ,
\end{align}
where $\delta = \sfrac{1}{100}$. Notably, the estimate~\eqref{eq:intermediatemodes} controls the growth forward-in-time of the `intermediate modes'. We can interpolate the estimate~\eqref{eq:decayestimatetilde} with the smoothing estimate~\eqref{eq:backwardsmoothing} to obtain
\begin{equation}
    \label{eq:smoothingestimatetilde}
    |e^{L_k t} (\tilde{P}^u_{(\sfrac{3}{2} + 3 \delta) \gamma} w)_k| \lesssim e^{(\sfrac{3}{2} + \delta) \gamma t + \nu'|k| t} |(\tilde{P}^u_{(\sfrac{3}{2} + 3 \delta) \gamma} w)_k| \, , \quad t \leq 0 \, ,
\end{equation}
for some $\nu' > 0$ independent of $\gamma$. Notice that we do not derive a smoothing estimate for the intermediate modes. Rather, we exploit that
\begin{equation}
	\label{eq:substitute}
e^{\sfrac{9\gamma}{5} |t|} \left|\int_{-\infty}^t e^{(\sfrac{3}{2} + 4 \delta) \gamma(t-s)} f_k(s) \, ds \right| \lesssim \gamma^{-1} \| e^{\sfrac{9\gamma}{5} |\cdot| } f_k \|_{L^\infty(-\infty,t)} \, ,
\end{equation}
and the wavenumbers associated to the intermediate modes satisfy $\langle k \rangle \lesssim \gamma$. Hence,\\ \eqref{eq:substitute} will substitute for the smoothing estimate for intermediate modes. The center-stable directions will be estimated in the same way as previously.

We seek solutions $w$ to the integral equation
\begin{align}
&w(t) = e^{ tL } b_0 - \int_t^{0} e^{(t-s)L} \tilde{P}^u_{(\sfrac{3}{2} + 3 \delta) \gamma} \tilde{F}(w,\p_x w)(s) \,ds \nonumber\\
&\quad\quad\quad\quad\quad + \int_{-\infty}^{t} e^{(t-s)L} \tilde{P}^{cs}_{(\sfrac{3}{2} + 3 \delta) \gamma} \tilde{F}(w,\p_x w)(s) \, ds \, .
\end{align}
If we define $v := w - e^{ \cdot L } b_0$, we can more directly solve
\begin{align}
v(t) &= \Phi(v) := - \int_t^{0} e^{(t-s)L} \tilde{P}^u_{(\sfrac{3}{2} + 3 \delta) \gamma} \tilde{F}(v+e^{\cdot L} b_0,\p_x (v+e^{\cdot L} b_0))(s) \,ds \nonumber\\
&\quad + \int_{-\infty}^{t} e^{(t-s)L} \tilde{P}^{cs}_{(\sfrac{3}{2} + 3 \delta) \gamma} \tilde{F}(v+e^{\cdot L} b_0,\p_x (v+e^{\cdot L} b_0))(s) \, ds
\end{align}
in the function space $\tilde{C}_{\sfrac{3}{2}\gamma}(\R_-;A^{0,1})$. We demonstrate how to prove that $\Phi$ stabilizes a ball; it is not much more difficult to establish that $\Phi$ is a contraction. We estimate the modes in three cases, beginning with the strongly unstable modes:
\begin{align}
&\| \int_t^{0} e^{(t-s)L} \tilde{P}^u_{(\sfrac{3}{2} + 3 \delta) \gamma} \tilde{F}(v+e^{\cdot L} b_0,\p_x (v+e^{\cdot L} b_0))(s) \,ds \|_{\tilde{L}^\infty_{\sfrac{3}{2}\gamma}(\R_-;A^{0,1})} \nonumber\\
& \lesssim \| \tilde{F}(v+e^{\cdot L} b_0,\p_x (v+e^{\cdot L} b_0)) \|_{\tilde{L}^\infty_{\sfrac{9\gamma}{5}}(\R_-;A^{0,1})} \nonumber\\
& \lesssim \| e^{\cdot L} b_0 \|_{\tilde{L}^\infty_{\sfrac{9\gamma}{10}}(\R_-;A^{0,1})}^2 + \| v \|_{\tilde{L}^\infty_{\sfrac{3\gamma}{2}}(\R_-;A^{0,1})}^2 \, .
\end{align}
The main tool is the smoothing estimate~\eqref{eq:smoothingestimatetilde} in conjunction with~\eqref{eq:backwardevolution}. Next, we estimate the intermediate modes
\begin{align}
&\| \int_{-\infty}^{t} e^{(t-s)L} (\tilde{P}^{cs}_{(\sfrac{3}{2} + 3\delta)\gamma} - P^{cs}) \tilde{F}(v+e^{\cdot L} b_0,\p_x (v+e^{\cdot L} b_0))(s) \, ds \|_{\tilde{L}^\infty_{\sfrac{3\gamma}{2}}(\R_-;A^{0,1})} \nonumber\\
& \lesssim \| e^{\cdot L} b_0 \|_{\tilde{L}^\infty_{\sfrac{9\gamma}{10}}(\R_-;A^{0,1})}^2 + \| v \|_{\tilde{L}^\infty_{\sfrac{3\gamma}{2}}(\R_-;A^{0,1})}^2 \, .
\end{align}
The main tool is the forward-in-time growth estimate \!\eqref{eq:intermediatemodes} in conjunction with \!\eqref{eq:substitute}. Finally, we estimate the contribution of the center-stable modes,
\begin{equation}
\int_{-\infty}^{t} e^{(t-s)L} P^{cs} \tilde{F}(v+e^{\cdot L} b_0,\p_x (v+e^{\cdot L} b_0))(s) \, ds \, ,
\end{equation}
in the same way as before. We omit the remaining details.

This completes the proof of~\eqref{eq:asymptotics} and Theorem~\ref{thm:unstablemanifold}.

%\begin{remark} It is possible to close existence at finite regularity, so it is not necessary to require that $F$ is analytic (but it is necessary to get the ill-posedness to control high frequencies for long enough times); $C^{1,\alpha}$ should suffice.
%\end{remark}

\begin{corollary}[Ill-posedness]
\label{cor:illposednessfullynonlinear}
If $s>1$, then
for any $t< 0$ and $M>0$, there exists a sequence $\{a_0^{(n)} \}$ of unstable data such that $\| a_0^{(n)} \| _{A^{0,1}} \leq \varepsilon_0$, $\| a_0^{(n)} \| _{H^s} \geq M$, and the corresponding solutions $w^{(n)}$ of the PDE \eqref{eq:perturbedequation} satisfy 
\begin{equation} 
\| w^{(n)} ( \cdot, t ) \| _{H^s} \to 0  \text{ as } n \to +\infty \, .
\end{equation} 
\end{corollary}

\begin{proof} 
Fix $\mu \in \sigma (\bm{L})$. Assume that $\Re \mu >0$; otherwise, replace $n$ with $-n$ below. We consider a sequence of eigenvalues $\lambda_n =  n ( \mu + o_{n \to +\infty}(1))$ and a corresponding sequence of eigenfunctions $b_0^{(n)}$ such that
\begin{equation}
1 \lesssim \| b_0^{(n)} \|_{A^{0,1} } \leq \varepsilon_1 \, , \quad \| b_0^{(n)} \| _{ H^s} = 2M \, ,
\end{equation}
and $\{ w^{(n)} \}$ is the corresponding sequence of solutions with scattering data $b_0^{(n)}$ guaranteed by Theorem~\ref{thm:unstablemanifold}. We choose $b_0^{(n)}$ sufficiently small to guarantee that each $a_0^{(n)} := P^u w^{(n)}(0)$ satisfies $1 \lesssim \| a_0^{(n)} \|_{A^{0,1}} \leq \varepsilon_0$. In particular, \emph{since $b_0^{(n)}$ consists of a single Fourier mode}, we can ensure via~\eqref{eq:asymptotics} at $t=0$ that $\| a_0^{(n)} \|_{H^s} \geq M$.

We have the linear estimate 
\begin{equation} 
    \label{eq:thingtoovsert}
\| e^{t L } b_0^{(n)} \| _{H^s } \leq C  e^{\lambda_n t } \| b_0^{(n)} \| _{H^s} \, .
\end{equation} 
From \eqref{eq:asymptotics}, we have 
\begin{equation}  \sum \langle k \rangle | \widehat{w^{(n)}} ( k,t) - \widehat{ e^{t L } b_0^{(n)} } ( k , t ) | \leq C \varepsilon_1^2 e^{  \frac{ 3 \lambda_n t}{2}}. \end{equation} 
We then estimate
\begin{align} 
\| w^n (\cdot, t ) - e^{tL} b_0^{(n)} \|^2_{H^s} & = \sum \langle k \rangle  ^{2s} | \widehat{w^{(n)}} ( k,t) - \widehat { e^{t L } b_0^{(n)}} ( k , t ) |^2
\nonumber \\ &  \leq C \varepsilon_1^2 e^{\frac{ 3 \lambda_n t }{2}} \sum \langle k \rangle ^{2s-1 } ( | \widehat{w^{(n)}} ( k , t ) | +  |\widehat { e^{t L } b_0^{(n)}} ( k , t )| ) \nonumber\\
& \leq  C\varepsilon_1^2 e^{ \frac{ 3 \lambda_n t }{2}} 
 \sum e^{\zeta |t| k } (  | \widehat{w^{(n)}} ( k , t ) | +  |\widehat { e^{t L } b_0^{(n)}} ( k , t )| ) 
\nonumber\\ &  \leq C \varepsilon_0 \varepsilon_1^2 e^{\frac {m_0t}{2}  } e^{ \frac{ 3 \lambda_n t }{2}} \, ,
 \end{align} 
 where the last inequality follows from \eqref{eq:estimates}. Therefore,
 \begin{equation} 
 \| w^{(n)} ( \cdot, t ) \|  _{ H^s } \leq  \| w^{(n)} (\cdot, t ) - e^{tL} b_0^{(n)} \|_{H^s} + \| e^{t L } b_0^{(n)} \| _{H^s } \to 0 \text{ as } n \to +\infty \, .
 \end{equation} 
 \end{proof}

%%%%%%%%%%%%%%%%%%%%%%%%%%%%%%%%%%%%%%%%%%%%%%%%%%%%%%
%          7. REFERENCES SECTION
%%%%%%%%%%%%%%%%%%%%%%%%%%%%%%%%%%%%%%%%%%%%%%%%%%%%%%

%       READ THIS SECTION CAREFULLY

% Each of the references below MUST be cited in your article above. Do not include references that are not cited in your article.

% Follow the examples below carefully. We strongly suggest that you copy and paste your reference information directly into our examples.

% List all references in alphabetical order according to the first author's last name.

% Verify each URL works correctly and can be accessed properly. Your URL links should be to reputable websites. The command line for a website link begins with: \url{ }

% Do not add MR or DOI numbers to your references. AIMS production staff will add this information.

% Using BibTex is not recommended but can be handled.

%\bibliography{bibliography}

\begin{thebibliography}{99}

\bibitem{ambrose2019global} (MR4001469) [10.1512/iumj.2019.68.7721] 
\newblock D. M. Ambrose, J. L. Bona and T Milgrom, 
\newblock \doititle{Global solutions and ill-posedness for the {K}aup system and related {B}oussinesq systems}, 
\newblock {\em Indiana Univ. Math. J.}, \textbf{68} (2019), 1173-1198.

\bibitem{Arnold} (MR202082) 
\newblock V.~Arnold, 
\newblock Sur la g\'{e}om\'{e}trie diff\'{e}rentielle des groupes de {L}ie de dimension infinie et ses applications \`a l'hydrodynamique des fluides parfaits, 
\newblock {\em Ann. Inst. Fourier (Grenoble)}, \textbf{16} (1966), 319-361.

\bibitem{bahouri} (MR2768550) [10.1007/978-3-642-16830-7] 
\newblock H. Bahouri, J.-Y. Chemin and R. Danchin, 
\newblock \doititle{\emph{Fourier Analysis and Nonlinear Partial Differential Equations}}, volume 343,
\newblock Springer Science $\&$ Business Media, 2011.

\bibitem{ckn} (MR673830) [10.1002/cpa.3160350604] 
\newblock L.~Caffarelli, R.~Kohn and L.~Nirenberg, 
\newblock \doititle{Partial regularity of suitable weak solutions of the {N}avier-{S}tokes equations}, 
\newblock {\em Commun. Pure Appl. Math.}, \textbf{35} (1982), 771-831.

\bibitem{Caflisch1993Complex} (MR1234435) [10.1016/0167-2789(93)90195-7] 
\newblock R. E. Caflisch, 
\newblock \doititle{Singularity formation for complex solutions of the 3d incompressible euler equations}, 
\newblock {\em Phys. D}, \textbf{67} (1993), 1-18.

\bibitem{caflisch1986long} (MR859274) [10.1002/cpa.3160390605] 
\newblock R. E. Caflisch and O. F. Orellana, 
\newblock \doititle{Long time existence for a slightly perturbed vortex sheet}, 
\newblock {\em Commun. Pure Appl. Math.}, \textbf{39} (1986), 807-838.

\bibitem{caflisch1989singular} (MR982661) [10.1137/0520020] 
\newblock R. E. Caflisch and O. F. Orellana, 
\newblock \doititle{Singular solutions and ill-posedness for the evolution of vortex sheets}, 
\newblock {\em SIAM J. Math. Anal.}, \textbf{20} (1989), 293-307.

\bibitem{cheng2020stable} (MR4066033) [10.1016/j.jde.2019.10.042] 
\newblock H. Y. Cheng and R. de~la Llave, 
\newblock \doititle{Stable manifolds to bounded solutions in possibly ill-posed pdes}, 
\newblock {\em J. Differ. Equ.}, \textbf{268} (2020), 4830-4899.

\bibitem{constantin1985simple} (MR812343) [10.1002/cpa.3160380605] 
\newblock P. Constantin, P. D. Lax and A. Majda, 
\newblock \doititle{A simple one-dimensional model for the three-dimensional vorticity equation}, 
\newblock {\em Commun. Pure Appl. Math.}, \textbf{38} (1985), 715-724.

\bibitem{constantin1994singular} (MR1252829) [10.1063/1.868050] 
\newblock P. Constantin, A. J. Majda and E. G. Tabak, 
\newblock \doititle{Singular front formation in a model for quasigeostrophic flow}, 
\newblock {\em Phys. Fluids}, \textbf{6} (1994), 9-11.

\bibitem{de1990one} (MR1063199) [10.1007/BF01334750] 
\newblock S. De~Gregorio, 
\newblock \doititle{On a one-dimensional model for the three-dimensional vorticity equation}, 
\newblock {\em J. Statist. Phys.}, \textbf{59} (1990), 1251-1263.

\bibitem{de2009smooth} (MR2538946) [10.1007/s10884-009-9140-y] 
\newblock R. de~la Llave, 
\newblock \doititle{A smooth center manifold theorem which applies to some ill-posed partial differential equations with unbounded nonlinearities}, 
\newblock {\em J. Dynam. Differ. Equ.}, \textbf{21} (2009), 371-415.

\bibitem{desjardins2006nonlinear} (MR2245231) [10.1007/s10114-005-0559-8] 
\newblock B. Desjardins and E. Grenier, 
\newblock \doititle{On nonlinear {R}ayleigh--{T}aylor instabilities}, 
\newblock {\em Acta Math. Sinica}, \textbf{22} (2006), 1007-1016.

\bibitem{Douglis1952} (MR52666) [10.1002/cpa.3160050202] 
\newblock A. Douglis, 
\newblock \doititle{Some existence theorems for hyperbolic systems of partial differential equations in two independent variables}, 
\newblock {\em Commun. Pure Appl. Math.}, \textbf{5} (1952), 119-154.

\bibitem{DuchonRobertVortexSheet} (MR943940) [10.1016/0022-0396(88)90105-2] 
\newblock J. Duchon and R. Robert, 
\newblock \doititle{Global vortex sheet solutions of {E}uler equations in the plane}, 
\newblock {\em J. Differ. Equ.}, \textbf{73} (1988), 215-224.

\bibitem{EbinMarsden} (MR271984) [10.2307/1970699] 
\newblock D. G. Ebin and J. Marsden, 
\newblock \doititle{Groups of diffeomorphisms and the motion of an incompressible fluid}, 
\newblock {\em Ann. Math.}, \textbf{92} (1970), 102-163.

\bibitem{friedrichs1948nonlinear} (MR25659) [10.2307/2372200] 
\newblock K. Friedrichs, 
\newblock \doititle{Nonlinear hyperbolic differential equations for functions of two independent variables}, 
\newblock {\em Amer. J. Math.}, \textbf{70} (1948), 555-589.

\bibitem{gallagherasymptotics} (MR2032938) 
\newblock I.~Gallagher, D.~Iftimie and F.~Planchon, 
\newblock Asymptotics and stability for global solutions to the {N}avier-{S}tokes equations, 
\newblock {\em Ann. Inst. Fourier (Grenoble)}, \textbf{53} (2003), 1387-1424.

\bibitem{han2016ill} (MR3509003) [10.1007/s00205-016-0985-z] 
\newblock D. H. Kwan and T. T. Nguyen, 
\newblock \doititle{Ill-posedness of the hydrostatic {E}uler and singular {V}lasov equations}, 
\newblock {\em Arch. Ration. Mech. Anal.}, \textbf{221} (2016), 1317-1344.

\bibitem{Katobook} (MR203473) 
\newblock T. Kato, 
\newblock {\em Perturbation theory for linear operators}, 
\newblock Die Grundlehren der mathematischen Wissenschaften, Band 132, Springer-Verlag New York, Inc., New York, 1966.

\bibitem{kielhofer2011bifurcation} 
\newblock H. Kielh{\"o}fer, 
\newblock {\em Bifurcation theory: An introduction with applications to partial differential equations}, volume 156, 
\newblock Springer Science $\&$ Business Media, 2011.

\bibitem{lerner2010instability} (MR2597507) [10.1353/ajm.0.0096] 
\newblock N. Lerner, Y. Morimoto and C.-J. Xu, 
\newblock \doititle{Instability of the {C}auchy-{K}ovalevskaya solution for a class of nonlinear systems}, 
\newblock {\em Amer. J. Math.}, \textbf{132} (2010), 99-123.

\bibitem{lerner2018onset} (MR3801816) [10.4171/JEMS/788] 
\newblock N. Lerner, T. Nguyen and B. Texier, 
\newblock \doititle{The onset of instability in first-order systems}, 
\newblock {\em J. Euro. Math. Soc.}, \textbf{20} (2018), 1303-1373.

\bibitem{li2008complex} (MR2459306) 
\newblock D. Li and Y. G. Sinai, 
\newblock Complex singularities of the {B}urgers system and renormalization group method, 
\newblock in {\em Current Developments in Mathematics}, 2006, International Press of Boston, 2008.

\bibitem{LiSinai} (MR2390325) [10.4171/JEMS/111] 
\newblock D. Li and Y. G. Sinai, 
\newblock \doititle{Blow ups of complex solutions of the 3{D} {N}avier-{S}tokes system and renormalization group method}, 
\newblock {\em J. Euro. Math. Soc.}, \textbf{10} (2008), 267-313.

\bibitem{metivier2005remarks} (MR2127041) [10.1090/conm/368/06790] 
\newblock G. M{\'e}tivier, 
\newblock \doititle{Remarks on the well-posedness of the nonlinear {C}auchy problem}, 
\newblock in {\em Geometric Analysis of PDE and Several Complex Variables: Dedicated to Fran{\c{c}}ois Treves}, American Mathematical Society, Providence, RI, 2005.

\bibitem{morisse2018hyperbolicity} (MR3760173) [10.1016/j.jde.2018.01.011] 
\newblock B. Morisse, 
\newblock \doititle{On hyperbolicity and {G}evrey well-posedness. {P}art two: {S}calar or degenerate transitions}, 
\newblock {\em J. Differ. Equ.}, \textbf{264} (2018), 5221-5262.

\bibitem{morisse2020hyperbolicity} (MR4191389) [10.5802/alco.132] 
\newblock B. Morisse, 
\newblock \doititle{On hyperbolicity and {G}evrey well-posedness. {P}art one: the elliptic case}, 
\newblock {\em Ann. H. Lebesgue}, \textbf{3} (2020), 1195-1239.

\bibitem{ndoumajoud2020m} 
\newblock K. Ndoumajoud and B. Texier, 
\newblock On {M}{\'e}tivier's {L}ax-{M}izohata theorem and extensions to weak defects of hyperbolicity, {P}art one, 
\newblock \arXiv{2012.08222}, 2020.

\bibitem{ndoumajoud2021m} 
\newblock K. Ndoumajoud and B. Texier, 
\newblock On {M}{\'e}tivier's {L}ax-{M}izohata theorem and extensions to weak defects of hyperbolicity, {P}art two, 
\newblock \arXiv{2103.02401}, 2021.

\bibitem{Nirenberg} (MR322321) 
\newblock L.~Nirenberg. 
\newblock An abstract form of the nonlinear {C}auchy-{K}owalewski theorem. 
\newblock {\em J. Differ. Geom.}, \textbf{6} (1972), 561-576.

\bibitem{Ogden} 
\newblock W.~J. Ogden, 
\newblock \emph{A Complexified Model of the {N}avier-{S}tokes Equations with Fractional Dissipation}, 
\newblock Undergraduate honors thesis, University of Minnesota, 2020.

\bibitem{polavcik2008zeros} 
\newblock P. Pol{\'a}{\v{c}}ik and V. {\v{S}}ver{\'a}k, 
\newblock Zeros of complex caloric functions and singularities of complex viscous {B}urgers equation, 
\newblock 2008.

\bibitem{seregin2002local} (MR1891072) [10.1007/s00021-002-8533-z] 
\newblock G. A. Seregin, 
\newblock \doititle{Local regularity of suitable weak solutions to the {N}avier--{S}tokes equations near the boundary}, 
\newblock {\em J. Math. Fluid Mech.}, \textbf{4} (2002), 1-29.

\bibitem{VS} [10.5802/jedp.658]
\newblock V. {\v S}ver\'ak, 
\newblock \doititle{\emph{On Certain Models in the {PDE} Theory of Fluid Flows}}, 
\newblock Journ\'ees \'equations aux d\'eriv\'ees partielles, 2017.

\bibitem{sverak2022singularities} 
\newblock V. {\v S}ver{\'a}k, 
\newblock On singularities in the quaternionic {B}urgers equation, 
\newblock \emph{Ann. Math. Qu{\'e}bec}, (2022), 1-14.

\bibitem{TaoEulerArnold} 
\newblock T. Tao, 
\newblock The {E}uler-{A}rnold equation, 2010, 
\newblock URL:
\url{https://terrytao.wordpress.com/2010/06/07/the-euler-arnold-equation/}.
Last visited on July 16, 2023.

\bibitem{TaoAveragedNSE} (MR3486169) [10.1090/jams/838] 
\newblock T. Tao, 
\newblock \doititle{Finite time blowup for an averaged three-dimensional {N}avier-{S}tokes equation}, 
\newblock {\em J. Amer. Math. Soc.}, \textbf{29} (2016), 601-674.

\bibitem{TaoLagrangianModifications} (MR3595455) [10.1007/s40818-016-0019-z] 
\newblock T. Tao, 
\newblock \doititle{Finite time blowup for {L}agrangian modifications of the three-dimensional {E}uler equation}, 
\newblock {\em Ann. Partial Differ. Equ.}, \textbf{2} (2016), 79 pp.

\bibitem{Wu2006} (MR2230845) [10.1002/cpa.20110] 
\newblock S. J. Wu, 
\newblock \doititle{Mathematical analysis of vortex sheets}, 
\newblock {\em Commun. Pure Appl. Math.}, \textbf{59} (2006), 1065-1206.


\end{thebibliography}
%\bibliographystyle{plain}

\medskip
% The information below will be filled in by AIMS production staff.
Received for publication July 2023; early access October 2023.
\medskip

\end{document}